\documentclass[12pt]{amsart} 
\usepackage{amsmath,amssymb,amscd,amsthm}
\usepackage{latexsym}
\usepackage{graphicx}
\usepackage[english]{babel}
\usepackage[latin1]{inputenc}       
\usepackage{textcomp}
\usepackage{times}
\usepackage{pgf,pgfarrows,pgfnodes,pgfautomata,pgfheaps}
\usepackage{colortbl}
\usepackage{epigraph}

\usepackage[a4paper,twoside,left=3cm,right=2.8cm,top=3.1cm,bottom=2.3cm]{geometry}

\newtheorem{observation}{Observation}
\newtheorem{question}{Question}
\newtheorem{fact}{Fact}

\newtheorem{theorem}{Theorem}[section]
\newtheorem{corollary}[theorem]{Corollary}
\newtheorem{proposition}[theorem]{Proposition}
\newtheorem{lemma}[theorem]{Lemma}

\newtheorem{remark}[theorem]{Remark}

\def\irr#1{{\rm Irr}(#1)}
\def\irrr#1#2 {\irr {#1 \mid #2}}
\newcommand{\R}{\mathbb R}

\newcommand{\sfe}{{{\mathbb S}^{n-1}}}

\newcommand{\E}{\mathbb E}

\begin{document}

\title[On the maximal perimeter of convex sets with respect to probability measures]{Some remarks about the maximal perimeter of convex sets with respect to probability measures}
\author[Galyna V. Livshyts]{Galyna V. Livshyts}

\address{School of Mathematics, Georgia Institute of Technology} \email{glivshyts6@math.gatech.edu}

\subjclass[2010]{Primary: 52} 
\date{\today}
\begin{abstract}
In this note we study the maximal perimeter of a convex set in $\R^n$ with respect to various classes of measures. Firstly, we show that for a probability measure $\mu$ on $\R^n$, satisfying very mild assumptions, there exists a convex set of $\mu$-perimeter at least $C\frac{\sqrt{n}}{\sqrt[4]{Var|X|} \sqrt{\mathbb{E}|X|}}.$ This implies, in particular, that for any isotropic log-concave measure $\mu$ one may find a convex set of $\mu$-perimeter of order $n^{\frac{1}{8}}$. Secondly, we derive a general upper bound of $Cn||f||^{\frac{1}{n}}_{\infty}$ on the maximal perimeter of a convex set with respect to any log-concave measure with density $f$ in an appropriate position. 

Our lower bound is attained for a class of distributions including the standard normal distribution. Our upper bound is attained, say, for a uniform measure on the cube.

In addition, for isotropic log-concave measures we prove an upper bound of order $n^2$ for the maximal $\mu$-perimeter of a convex set.
\end{abstract}
\maketitle

\section{Introduction}

The surface area (perimeter) of a convex set $Q$ in $\R^n$ with respect to the measure $\mu$ is defined to be
\begin{equation}\label{outer}
\mu^+(\partial Q)=\liminf_{\epsilon\rightarrow +0}\frac{\mu((Q+\epsilon B_2^n)\backslash Q)}{\epsilon}.
\end{equation}

For various probability measures $\mu$, we shall study the quantity 
$$\Gamma(\mu)=\sup_{Q} \mu^+(\partial Q),$$
where the supremum runs over all convex sets in $\R^n.$ We shall also use notation $\Gamma(X)=\Gamma(\mu)$, where $X$ denotes a random vector distributed with respect to $\mu$. In other words,
\begin{equation}\label{def}
\Gamma(X)=\sup_Q\liminf_{\epsilon\rightarrow 0}\frac{P(X\in Q+\epsilon B_2^n)-P(X\in Q)}{\epsilon},
\end{equation}
where the supremum runs over all convex sets $Q$ in $\R^n$. This characteristic measures the sensitivity of the distribution of $X$ with respect to small perturbations. We will use both notations $\Gamma(X)$ and $\Gamma(\mu)$ interchangeably. We remark that this notation follows the notation of Nazarov \cite{fedia}.

Ball showed \cite{ball} that in the case when $\gamma$ is the standard Gaussian measure, $\Gamma(\gamma)\leq Cn^{\frac{1}{4}}$. It was shown by Nazarov \cite{fedia} that this estimate is in fact sharp, and there exists a convex set $Q\subset\R^n$ such that $\gamma^+(\partial Q)\geq 0.28 n^{\frac{1}{4}}$; furthermore, another argument was developed by Nazarov to obtain the upper bound of $0.67 n^{\frac{1}{4}}$, and a sharp estimate in the case of non-standard Gaussian measure was proved there as well. Bentkus \cite{Bent} used the bound on $\Gamma(\gamma)$ for the case of the standard Gaussian measure in order to give bounds on the rates of convergence in the multidimensional central limit theorem. Recently, Raic \cite{raic} extended Nazarov's results and provided further numerical improvements to Benkus's estimate. 

Further, estimates on the perimeter of surfaces given by level sets of polynomials with respect to the standard Gaussian measure were studied by Kane \cite{kane}: a dimension-free bound was obtained there.

The author studied $\Gamma(\mu)$ in \cite{GL}, \cite{GL1} and \cite{GL3}. In \cite{GL} it was shown that for any rotation invariant log-concave random vector $X,$ one has $\Gamma(X)\approx \frac{\sqrt{n}}{\sqrt[4]{Var|X|} \sqrt{\mathbb{E}|X|}};$ here $\approx$ stands for the equality up to a multiple of an absolute constant. In \cite{GL1} it was shown that for measures $\mu_p$ with density $C_{n,p}e^{-\frac{|x|^p}{p}}$ one has $\Gamma(\mu_p)\approx C(p)n^{\frac{3}{4}-\frac{1}{p}}$, for any $p\in (0,\infty)$; note that the case $p\in (0,1)$ corresponds to non-log-concave measures. In \cite{GL3}, the maximal perimeter of convex polytopes with respect to log-concave rotation-invariant measures was estimated.

In this note, we study $\Gamma(X)$ for general classes of distributions. Our first result is the following lower bound estimate.

\begin{theorem}\label{lb}
Let $X$ be a random vector in $\R^n$ with an absolutely continuous distribution. Suppose that $\sqrt{Var(|X|)}\leq \alpha\E|X|$, for $\alpha\in [0,1)$. Then
\begin{equation}\label{lb-eq}
\Gamma(X)\geq C\frac{\sqrt{n}}{\sqrt[4]{Var(|X|)} \sqrt{\mathbb{E}|X|}},
\end{equation}
where $C>0$ depends only on $\alpha.$ Namely, $C=C(\alpha)\rightarrow_{\alpha\rightarrow 0} 0.06$, and $C(\alpha)\rightarrow_{\alpha\rightarrow 1} 0$ .
\end{theorem}

As was pointed out earlier, this lower bound is sharp, in particular, for all log-concave rotation invariant random vectors. It is, however, not always sharp: for example, if $X$ is distributed uniformly on the cube of unit volume, then $\Gamma(X)=2n$, while the right hand side of (\ref{lb-eq}) is of order $n^{\frac{1}{4}}$.

Note that the left hand side of (\ref{lb-eq}) is shift-invariant, while the right-hand side is not, and therefore one might further strengthen the statement by considering optimizing by shifts. Furthermore, a more general statement, Theorem \ref{general} shall be formulated in Section 2.

Recall that a measure $\mu$ on $\R^n$ is called log-concave if for any pair of Borel measurable sets $A,B\subset\R^n$, and any $\lambda\in [0,1]$, one has $\mu(\lambda A+(1-\lambda)B)\geq \mu(A)^{\lambda}\mu(B)^{1-\lambda}.$ Recall also that a probability measure $\mu$ is called isotropic if its covariance matrix is identity, and the barycenter is at the origin.

The quantity $\sqrt{Var(|X|)}$ is of central importance in the study of the geometry of log-concave measures. The celebrated Thin Shell Conjecture (Anttila, Ball, Perissinaki \cite{ABP}, Bobkov, Koldobsky \cite{BK}) states that for all isotropic log-concave measures one has $Var(|X|)\leq C_0,$ with a dimension-independent constant $C_0$. The strong connection of this question with Bourgain's hyperplane conjecture \cite{bour1}, \cite{bour} was found by Eldan and Klartag \cite{ElKl}. It was shown by Eldan \cite{Eldan} that the Thin Shell Conjecture implies the celebrated KLS-conjecture of Kannan, Lovasz, Simonovits \cite{KLS} up to a polylog. The best currently known bound for $\sqrt{Var(|X|)}$ is of order $n^{\frac{1}{4}}$, as was shown recently by Lee and Vempala \cite{LV}, improving upon the previous record by Guedon and Milman \cite{GM}, which in turn improved upon the results of Fleury \cite{fl} and Klartag \cite{Kl}, \cite{Kl2}.

In view of the result of Lee and Vempala \cite{LV}, and in view of the fact that $\E|X|=(1+o(1))\sqrt{n}$ for an isotropic log-concave vector $X,$ we formulate the following corollary of Theorem \ref{lb}:  

\begin{observation} For any isotropic log-concave vector in $\R^n$, $\Gamma(X)\geq Cn^{\frac{1}{8}},$ for an absolute constant $C.$
\end{observation}

In addition, we would like to point out the following, perhaps rather hypothetical

\begin{observation} In case there existed a log-concave isotropic measure $\mu$ for which 
$$\mu^+(\partial Q)\leq o(1)\cdot n^{\frac{1}{4}}$$ 
for all convex sets $Q$ in $\R^n$, then the Thin Shell Conjecture would not be true, and this measure would be a counter-example.
\end{observation}

In other words, in case the Thin Shell Conjecture holds, then the Gaussian measure is among the minimizers of $\Gamma(\mu)$ in the class of isotropic log-concave measures, up to a constant multiple.

Lastly, we point out another class of measures (in addition to the rotation invariant log-concave ones), for which the lower bound from Theorem \ref{lb} is sharp. Recall that a measure is called unconditional if it is invariant under the reflections with respect to coordinate hyperplanes.

\begin{observation}\label{hessian}
Suppose $\mu$ is an isotropic log-concave unconditional measure on $\R^n$ with twice differentiable density $f$, such that
\begin{equation}\label{hess}
0\leq -\nabla^2 \log f\leq Id,
\end{equation}
in the matrix sense. Then there exist absolute constants $C_1\geq C_2>0$ such that
\begin{equation}\label{concl} 
C_2n^{\frac{1}{4}}\leq \Gamma(\mu)\leq C_1 n^{\frac{1}{4}}.
\end{equation}
\end{observation}

Indeed, recall that
$$\mu^+(Q)=\int_{\R^n} |\nabla 1_Q(x)| d\mu(x),$$
where the equality is understood in the sense of smooth approximation of $1_Q$. The celebrated theorem of Caffarelli \cite{caf} states, in particular, that the assumption (\ref{hess}) implies that there exists a contraction map transporting $\mu$ onto the standard Gaussian measure $\gamma$, which is a gradient of a convex function (see also Theorem 2.2 and Corollary 6.4 in the nice survey by Kolesnikov \cite{Koles}, where the most general statements can be found). Therefore, a standard transportation argument shows that for every convex set $Q$ there exists a convex set $K$, such that $\mu^{+}(\partial Q)\leq \gamma^+(\partial K)$, where $\gamma$ is the Gaussian measure. The upper bound hence follows from the estimate of Ball \cite{ball} and Nazarov \cite{fedia}, discussed earlier. As for the lower estimate, Klartag \cite{Kl3} showed that $Var(|X|)\leq C_0$ in the case of unconditional log-concave distributions, and therefore, in view of the fact that $\E |X|=(1+o(1))\sqrt{n}$ for isotropic log-concave measures, Theorem \ref{lb} implies the lower bound in (\ref{concl}). 

\begin{remark} At least in the case of product measures, the upper bound from Observation \ref{hessian} can be also derived by directly adapting the argument of Nazarov \cite{fedia}, without invoking Caffarelli's theorem.
\end{remark}

\medskip
\medskip

Next, we shall switch to the discussion of the upper bound for the quantity $\Gamma(\mu)$ for certain classes of measures. 


We observe that for any $\alpha>0,$ one has from (\ref{def}) that $\Gamma(\alpha X)=\frac{1}{\alpha}\Gamma(X)$. In other words, for any measure $\mu$ with density $f,$ the quantity 
$$||f||^{-\frac{1}{n}}_{\infty}\cdot \Gamma(\mu)$$
is invariant under the transformations of the type $f_{\alpha}(x)=\alpha^n f(\alpha x)$. Therefore, it makes sense to estimate $\Gamma(\mu)$ in terms of $||f||_{\infty}^{\frac{1}{n}}$.

Recall that a density $f$ is called unimodule if all its level sets are convex. Note that the class of unimodule densities includes all log-concave densities. We shall show the following

\begin{theorem}\label{ub}
Let $X$ be a random vector on $\R^n$ with an absolutely continuous unimodule density $f$. Then there exists a linear volume preserving transformation $T$ such that
$$\Gamma(TX)\leq Cn ||f||^{\frac{1}{n}}_{\infty},$$
where $C>0$ is an absolute constant.
\end{theorem}

Let us emphasize here that while $||f||_{\infty}$ is invariant under volume preserving transformations, $\Gamma(TX)$ is not. Theorem \ref{ub} should be considered as a generalization of the result of Ball \cite{Ball}, where it was proved in the case of uniform distributions on convex sets. Our proof also relies on Ball's volume ratio estimate from \cite{ball_cube1}.

We would like to continue the discussion about $\Gamma(\mu)$ for log-concave measures in the isotropic position, and not just the position that arises in Theorem \ref{ub}. For instance, the lower bound from Theorem \ref{lb} is biggest possible when the measure is in the isotropic position (modulo scaling). Some interesting estimates of the reverse isoperimetric type were obtained by Markessinis, Paouris, Saroglou \cite{MPS}, where the authors show that perimeters of convex sets and related quantities may change significantly when the position is changed from isotropic to an optimal one.

Below we state a polynomial upper bound on the maximal perimeter in the isotropic case.

\begin{proposition}\label{ub-lc}
Let $X$ be an isotropic log-concave random vector. Then
$$\Gamma(X)\leq Cn^2,$$
where $C>0$ is an absolute constant.
\end{proposition}

In the case when $\mu$ is uniform on an isotropic convex set $K$, one may easily show that the upper bound for $\Gamma(\mu)$ is in fact $Cn$ (see below (\ref{oneline})). 
In addition, analyzing the proof of Theorem \ref{ub}, one may observe that $\Gamma(\mu)\leq Cn$ for isotropic log-concave measures with homogenous level sets, as well as for 1-symmetric log-concave measures. We suspect that the correct bound in Proposition \ref{ub-lc} in general should be $Cn$ rather than $Cn^2$. Such bound would follow from an affirmative answer to Question \ref{question}, formulated in the third section.

In Section 2 we derive Theorem \ref{lb}, as well as additional more general estimates. In Section 3 we derive Theorem \ref{ub} and Proposition \ref{ub-lc}, and make some further observations.

\textbf{Acknowledgement.} The author is supported by the NSF CAREER DMS-1753260. 

\section{Proof of Theorem \ref{lb}.}

We follow the idea from Nazarov \cite{fedia}, where it was applied in the case of Gaussian measure $\mu$. Certain modifications to that argument are made: we shall use properties of independent random variables in a stronger way, which will enable us to formulate a simple proof of the more general statement of Theorem \ref{lb}.

The notation $P(\cdot)$ is used for probability, and $\E$ stands for expectation everywhere below; occasionally, sub-indexes will be used to emphasize the distribution with respect to which the probability or the expectation is taken.

We begin with some technical lemmas.

\begin{lemma}\label{gauss_property}
Let $Y$ be a random vector distributed according to the standard Gaussian distribution. Then, for any $y\in\R^n,$ for any $\rho>0,$
$$\lim_{\epsilon\rightarrow 0} \frac{P_Y(\langle y,Y\rangle\in[\rho,\rho+\epsilon |Y|])}{\epsilon}=\frac{1+o(1)}{\sqrt{2\pi}}\frac{\sqrt{n}}{|y|}e^{-\frac{\rho^2}{2|y|^2}}.$$
\end{lemma}
\begin{proof} Recall that for a standard normal vector $Y$, the random vector $\frac{Y}{|Y|}$ and the random variable $|Y|$ are independent. Recall as well that $\langle \frac{y}{|y|},Y\rangle\sim N(0,1)$, for any $y\in\R^n\setminus 0$. For $y=0$ the conclusion of the lemma is straightforward anyway. Therefore,
\begin{equation}\label{tocont}
\lim_{\epsilon\rightarrow 0} \frac{P_Y(\langle y,Y\rangle\in[\rho,\rho+\epsilon |Y|])}{\epsilon}=\frac{1}{|y|}\lim_{\delta\rightarrow 0} \E_{\theta}\frac{ P_Z(Z\in [\frac{\rho}{y}, \frac{\rho}{|y|}+\delta |\theta|])}{\delta},
\end{equation}
where $Z$ is the standard normal random variable and $\theta$ is the standard normal random vector in $\R^n$ independent of $Z.$ Above we used the change of variables $\delta =\frac{\epsilon}{|y|}$.

Next, by the monotone convergence theorem and Fact \ref{tech1} from appendix, we may exchange $\E_{\theta}$ and the limit, and get that (\ref{tocont}) equals
$$\frac{1}{|y|}E_{\theta} |\theta|\lim_{\nu\rightarrow 0} \frac{P_Z(Z\in [\frac{\rho}{y}, \frac{\rho}{|y|}+\nu])}{\nu},$$
where the change of variable $\nu=\delta |\theta|$ was used. It remains to recall that $E_{\theta}|\theta|=(1+o(1))\sqrt{n}$, and that the density of $Z$ is $\frac{1}{\sqrt{2\pi}}e^{-\frac{s^2}{2}}$, to finish the proof.
\end{proof}

Next, we shall derive a convenient equivalent formulation for the surface area of a polytope.

\begin{lemma}\label{technical}
Fix an integer $N$, a collection of vectors $\{\theta_i\}_{i=1}^N$ and a collection of non-negative real numbers $\{\rho_i\}_{i=1}^N$. Let $Q$ be a polytope given by
$$Q=\cap_{i=1}^N\{y:\,\langle y, \theta_i\rangle\leq \rho_i\},$$ 
and let $\mu$ be any measure on $\R^n$ with bounded absolutely continuous density. Denote by $X$ the random vector distributed according to $\mu$. Then
$$\mu^+(\partial P)=\sum_{i=1}^N\lim_{\epsilon\rightarrow 0} \frac{1}{\epsilon} P_X\left(\langle X,\theta_i\rangle\in[\rho_i,\rho_i+\epsilon |\theta_i|],\,\forall j\neq i:\,\langle X,\theta_j\rangle\leq\rho_j\right).$$ 
\end{lemma}
\begin{proof} Let us use the notation
$$F_i^{\epsilon}=\{x\in\R^n:\,\, \langle x,\theta_i\rangle\in[\rho_i,\rho_i+\epsilon |\theta_i|],\,\forall j\neq i:\,\langle x,\theta_j\rangle\leq\rho_j\}.$$
It is well-known (see, e.g., Schneider \cite{book4}), that for Lebesgue measure $\lambda$ and a polytope $Q$ as in the lemma,
$$\frac{1}{\epsilon}\lambda\left((Q+\epsilon B_2^n)\setminus \left(\cup_{i=1}^N F_i^{\epsilon}\right) \right)\rightarrow_{\epsilon\rightarrow 0} 0.$$
Consequently, the same holds for any measure $\mu$ with a bounded absolutely continuous density. The Lemma hence follows from the definition of $\mu^+(\partial Q).$
\end{proof}

\begin{remark} Note that for our argument, it is in fact sufficient to have a lower estimate for $\mu^+(\partial Q)$, and thus enough to only use the straightforward observation
$$ \cup_{i=1}^N F_i^{\epsilon}\subset Q+\epsilon B_2^n,$$
in place of Lemma \ref{technical}.
\end{remark}

\subsection{Proof of Theorem \ref{lb}.} For an integer $N$ (to be chosen later), consider i.i.d. standard Gaussian random vectors $Y_1,...,Y_N.$ Fix $\rho>0,$ to be chosen later. Consider a random polytope 
$$Q:=\cap_{k=1}^N \{y:\,\langle Y_i, y\rangle\leq \rho\}.$$ 
Without loss of generality we may assume that the density of $\mu$ is bounded (see appendix for the details). By Lemma \ref{technical}, 
\begin{equation}\label{step1}
\mathbb{E}_Y \mu^+(\partial Q)=\sum_{i=1}^N\mathbb{E}_Y \lim_{\epsilon\rightarrow 0} \frac{1}{\epsilon} \mathbb{E}_X1_{\{\langle X,Y_i\rangle\in[\rho,\rho+\epsilon |Y_i|],\,\forall j\neq i:\,\langle X,Y_j\rangle\leq\rho\}},
\end{equation}
where by $\E_Y$ we denote the expectation with respect to the joint distribution of $Y_1,...,Y_N$, and by $\E_X$ -- the expectation with respect to the random vector $X,$ distributed according to the measure $\mu.$ Under additional assumptions which do not affect the generality of the Theorem, we may interchange the limit and the expectation $\E_Y$, the expectations in $X$ and $Y$, and the limit and the expectation $\E_X$: the justification is outlined in Lemma \ref{justif} in the appendix. The right hand side of (\ref{step1}) equals to
$$
N\mathbb{E}_X\lim_{\epsilon\rightarrow 0} \frac{1}{\epsilon} \mathbb{E}_Y 1_{\{\langle X,Y_i\rangle\in[\rho,\rho+\epsilon |Y_i|],\,\forall j\neq i:\,\langle X,Y_j\rangle\leq\rho\}}=
$$ 
\begin{equation}\label{step2} 
N\mathbb{E}_X\lim_{\epsilon\rightarrow 0} \frac{P_Y(\langle X,Y_1\rangle\in[\rho,\rho+\epsilon |Y_1|])}{\epsilon}\cdot\prod_{j=2}^N P_Y(\langle X,Y_j\rangle\leq\rho),
\end{equation}
where in the last passage we used independence of $Y_1,...,Y_N$. Recall that $Y_i$ are i.i.d. standard Gaussian, and therefore, for any fixed vector $X$, 
\begin{equation}\label{gaussdist}
\langle X/|X|,Y_j\rangle\sim N(0,1).
\end{equation}
Hence, by Lemma \ref{gauss_property}
\begin{equation}\label{1}
\lim_{\epsilon\rightarrow 0} \frac{P_Y(\langle X,Y_1\rangle\in[\rho,\rho+\epsilon |Y_1|])}{\epsilon}=\frac{1+o(1)}{\sqrt{2\pi}}\frac{\sqrt{n}}{|X|}e^{-\frac{\rho^2}{2|X|^2}}.
\end{equation}
Next, in view of (\ref{gaussdist}) and the well-known inequality 
$$\int_a^{\infty} e^{-\frac{s^2}{2}}ds\leq \frac{1}{a} e^{-\frac{a^2}{2}},$$
we get
\begin{equation}\label{2}
P_Y(\langle X,Y_j\rangle\leq\rho)\geq 1-\frac{|X|}{\sqrt{2\pi}\rho}e^{-\frac{\rho^2}{2|X|^2}}.
\end{equation}

Combining (\ref{step2}), (\ref{1}) and (\ref{2}) we get

\begin{equation}\label{step3}
\mathbb{E}_Y \mu^+(\partial Q)\geq \frac{1+o(1)}{\sqrt{2\pi}}\E_X \bigl[\frac{\sqrt{n}}{|X|}e^{-\frac{\rho^2}{2|X|^2}}\cdot N\cdot\left(1-\frac{|X|}{\sqrt{2\pi}\rho}e^{-\frac{\rho^2}{2|X|^2}}\right)^{N-1}\bigr].
\end{equation}

For brevity, let $W=\sqrt{Var(|X|)}=\sqrt{\E(X-\E|X|)^2}$. By Markov's inequality we have, for any $\beta>0,$
$$P\left(|X|\in [\E|X|-\beta W,\E|X|+\beta W]\right)\geq 1-\frac{1}{\beta^2}.$$
Conditioning on the latter event, we see that (\ref{step3}) implies, for any $\beta<\frac{1}{\alpha},$
$$\mathbb{E}_Y \mu^+(\partial Q)\geq$$
\begin{equation}\label{step4}
 \frac{1+o(1)}{\sqrt{2\pi}}\left(1-\frac{1}{\beta^2}\right)\frac{\sqrt{n}}{\E|X|+\beta W}e^{-\frac{\rho^2}{2(\E|X|-\beta W)^2}}\cdot N\cdot\left(1-\frac{\E|X|+\beta W}{\sqrt{2\pi}\rho}e^{-\frac{\rho^2}{2(\E|X|+\beta W)^2}}\right)^{N-1}.
\end{equation}
Next, we let 
$$N=\bigl[\frac{\sqrt{2\pi}\rho}{\E|X|+\beta W}e^{\frac{\rho^2}{2(\E|X|+\beta W)^2}}+1\bigr]+1,$$
in which case
\begin{equation}\label{expon}
N\cdot\left(1-\frac{\E|X|+\beta W}{\sqrt{2\pi}\rho}e^{-\frac{\rho^2}{2(\E|X|+\beta W)^2}}\right)^{N-1}\geq \frac{1}{e}.
\end{equation}
We note that (\ref{step4}) and (\ref{expon}) yield, with this choice of $N:$
\begin{equation}\label{step5}
\mathbb{E}_Y \mu^+(\partial Q)\geq \frac{1+o(1)}{e}\left(1-\frac{1}{\beta^2}\right)e^{-\frac{\rho^2}{2(\E|X|-\beta W)^2}}\cdot \frac{\sqrt{n}\rho}{(\E|X|+\beta W)^2}e^{\frac{\rho^2}{2(\E|X|+\beta W)^2}}.
\end{equation}
Let $w=\frac{W}{\E|X|}$. By our assumption,
\begin{equation}\label{small-c}
\sqrt{Var(|X|)}\leq \alpha\E|X|,
\end{equation}
for some $\alpha<1,$ and therefore 
\begin{equation}\label{w}
w\in [0,\alpha]. 
\end{equation}
Let
\begin{equation}\label{rho}
\rho:=\frac{1}{2}\sqrt{\frac{1-(\beta w)^2}{\beta}}\cdot \frac{\E|X|}{\sqrt{w}}.
\end{equation}
In that case, observe, using (\ref{w}):
\begin{equation}\label{comp} 
-\frac{\rho^2}{2(\E|X|-\beta W)^2}+\frac{\rho^2}{2(\E|X|+\beta W)^2}=\frac{\rho^2}{2\E|X|^2}\frac{-4\beta w}{(1-\beta w)^2(1+\beta w)^2}\geq -0.5.
\end{equation}

Plugging the value of $\rho$ from (\ref{rho}) we have, by (\ref{step5}) and (\ref{comp}):
\begin{equation}\label{step6}
\mathbb{E}_Y \mu^+(\partial Q)\geq \frac{1+o(1)}{e\sqrt{e}}\left(1-\frac{1}{\beta^2}\right)\frac{\sqrt{n}\rho}{(\E|X|+W)^2}=
\end{equation}
$$\frac{1+o(1)}{2e\sqrt{e}}\left(1-\frac{1}{\beta^2}\right)\sqrt{\frac{1-(\beta\alpha)^2}{\beta}}\frac{1}{(1+\alpha)^2}\frac{\sqrt{n}}{\sqrt{\E|X|}\sqrt[4]{Var(|X|)}}.$$
It remains to optimize in $\beta$ to finish the proof $\square.$

\begin{remark} A ``non-probabilistic'' version of the proof (where the independence of $X_i$ is not used in a strong way) could be carried out as well, with the help of ideas and results from Sodin \cite{sodin}, however such argument is both longer and imposes additional assumptions.
\end{remark}

Furthermore, a more general result holds, with the same proof:

\begin{theorem}\label{general}
Let $X$ be a random vector such that 
$$P(|X+y|\in [a,b])\geq 1-\delta,$$
for some $0<a<b,$ some $\delta\in (0,1)$ and some vector $y\in \R^n.$ Then 
$$\Gamma(X)\geq \frac{C(1-\delta)\sqrt{n}}{b\sqrt{(b/a)^2-1}}.$$
\end{theorem}

Consequently, one may always have the following rough but general estimate:

\begin{corollary}
For any random vector $X$ with density $f$, 
$$\Gamma(X)\geq \sup_{y\in\R^n}\frac{Cn}{(\E|X+y|)^2||f||_{\infty}^{\frac{1}{n}}}.$$
\end{corollary}
\begin{proof} Note that in the case when the density of $X$ is unbounded, or in the case when the first moment of $X$ does not exist, the conclusion of the corollary is straightforward. Assume that $||f||_{\infty}<\infty$ and $\E|X|<\infty.$ 

Fix an arbitrary $y\in\R^n$. Observe that for any $a>0,$
\begin{equation}\label{smb}
P(|X+y|<a)=P(X\in -y+a B_2^n)\leq ||f||_{\infty}\cdot |-y+a B_2^n|=\left(c\frac{a||f||^{\frac{1}{n}}_{\infty}}{\sqrt{n}}\right)^n.
\end{equation}
Note also, by Markov's inequality:
\begin{equation}\label{ld}
P(|X+y|>4\E|X+y|)\leq 0.25.
\end{equation}
Therefore, we have
$$P(|X+y|\in [a,b])\geq 1-\delta,$$
with $\delta=0.5,$ $a=\frac{c\sqrt{n}}{||f||_{\infty}^{\frac{1}{n}}}$ and $b=4\E|X+y|.$ An application of Theorem \ref{general} together with the inequality
$$\sqrt{(b/a)^2-1}\leq \frac{b}{a}$$
yields the conclusion.
\end{proof}

\section{Upper bounds on maximal perimeter}

We begin with formulating a simple lemma that will be used a number of times; this trick was used before, for example, by Ball \cite{ball}.

\begin{lemma}\label{simple}
Let $K$ be a convex body. Suppose, for some $x_0\in\R^n$ and some $R>0,$ we have $R B_2^n+x_0\subset K.$ Then $|\partial K|_{n-1}\leq \frac{n|K|}{R}$.
\end{lemma}
\begin{proof} Our assumption implies that $\epsilon B_2^n\subset \frac{\epsilon}{R}K+y_0$, for some $y_0\in\R^n.$ Therefore,
$$|\partial K|_{n-1}=\lim_{\epsilon\rightarrow 0}\frac{|K+\epsilon B_2^n|-|K|}{\epsilon}
\leq\lim_{\epsilon\rightarrow 0}\frac{|K+\frac{\epsilon}{R}K+y_0|-|K|}{\epsilon}=$$
$$|K|\cdot\lim_{\epsilon\rightarrow 0}\frac{(1+\epsilon/R)^n-1}{\epsilon}=\frac{n|K|}{R}.$$
\end{proof}

\begin{remark} In the case when $\mu$ is uniform on an isotropic set $K$, in view of Lemma \ref{simple} and the fact that an isotropic convex body contains a ball of radius $0.1$ (see below Lemma \ref{klartag}), we note
\begin{equation}\label{oneline}
\Gamma(\mu)=\frac{|\partial K|_{n-1}}{|K|}\leq 10n.\end{equation}
\end{remark}

One can provide an integral formula for $\mu^+(\partial Q)$, for any absolutely continuous measure $\mu$ (see, e.g, \cite{GL}):
\begin{equation}\label{def2}
\mu^+(\partial Q)=\int_{\partial Q} f(y)d\sigma(y).
\end{equation}


Fix a unimodule probability measure $\mu$ with density $f(x)$. Denote the convex level set, that depends on $f$ and a parameter $t>0$, as follows:
$$K_{t}(f):=\{x\in \R^n:\, f(x)\geq t\}.$$
Define also 
$$R_t(f):=\sup_{y\in\R^n}\{r>0:\, rB_2^n+y\subset K_t(f)\}$$ 
to be the radius of the largest ball contained inside $K_t(f)$. We prove the following

\begin{lemma}\label{ub-general}
Let $\mu$ be a unimodule measure with absolutely continuous density $f$. Then for any convex set $Q,$
$$\mu^+(\partial Q)\leq n\cdot\inf_{t\in (0,||f||_{\infty})}\frac{||f||_{\infty}|K_t(f)|+||f||_1}{R_t(f)}.$$
\end{lemma}
\begin{proof} We will use notation $K_t(f)=K_t$ and $R_t(f)=R_t$. Fix an arbitrary $t\in (0,||f||_{\infty})$. For a convex body $Q$, observe that
\begin{equation}\label{ub1}
\mu^+(\partial Q\cap K_t)\leq |\partial Q\cap K_t|_{n-1}\cdot ||f||_{\infty}.
\end{equation}
Recall that the usual (Lebesgue) surface area of a convex body is smaller than that of any convex body containing it. Therefore,
\begin{equation}\label{ub2}
|\partial Q\cap K_t|_{n-1}\leq |\partial K_t|_{n-1}.
\end{equation}
Next, by Lemma \ref{simple},
\begin{equation}\label{ub3}
|\partial K_t|_{n-1}\leq \frac{n|K_t|}{R_t}.
\end{equation}

Combining (\ref{ub1}), (\ref{ub2}) and (\ref{ub3}), we get
\begin{equation}\label{est1}
\mu^+(\partial Q\cap K_t)\leq \frac{n||f||_{\infty}|K_t|}{R_t}.
\end{equation}

Next, consider $\mu^{+}(\partial Q\setminus K_t)$. Note that for any $a>0$ one has 
$$a=\int_0^{\infty} 1_{\{a\geq t\}}(t)dt.$$
Applying this observation with with $a=f(y)$, we write, in view of (\ref{def2}): 
$$\mu^{+}(\partial Q\setminus K_t)=\int_{\partial Q\setminus K_t} f(y)d\sigma(y)=$$
\begin{equation}\label{ub5}
\int_{\partial Q\setminus K_t} \int_0^{\infty} 1_{K_s} ds d\sigma(y)=\int_{0}^t |\partial Q\cap K_s|_{n-1} ds,
\end{equation}
where in the second passage we used Fubbini's theorem. As before, by convexity, we notice
\begin{equation}\label{ub6}
|\partial Q\cap K_s|_{n-1}\leq |\partial K_s|_{n-1}.
\end{equation}
For any $s\in [0,t]$, we have $K_t\subset K_s,$ and hence $R_s\geq R_t.$ Therefore, by Lemma \ref{simple},
\begin{equation}\label{ub7}
|\partial K_s|_{n-1}\leq  \frac{n|K_s|}{R_s}\leq \frac{n|K_s|}{R_t}.
\end{equation}
It remains to recall that
\begin{equation}\label{ub8}
\int_0^t |K_s|ds\leq \int_0^{\infty} |K_s|ds=||f||_1.
\end{equation}
Combining (\ref{ub5}), (\ref{ub6}), (\ref{ub7}) and (\ref{ub8}), we get
\begin{equation}\label{est2}
\mu^{+}(\partial Q\setminus K_t)\leq \frac{n||f||_1}{R_t}.
\end{equation}

Finally, (\ref{est1}) and (\ref{est2}) yield
$$\mu^+(\partial Q)= \mu^{+}(\partial Q\setminus K_t)+\mu^+(\partial Q\cap K_t)\leq n\frac{||f||_{\infty}|K_t|+||f||_1}{R_t}.$$
It remains to note that $t$ was arbitrary, and the proof is complete.
\end{proof} 

\begin{remark}
In view of the definition of $K_t$, observe that for any probability measure $\mu,$
\begin{equation}\label{ub4}
|K_t|\leq \frac{\mu(K_t)}{t}\leq \frac{1}{t},
\end{equation}
where we used the simple lower bound for $\mu(K_t)$. Therefore, for a unimodule probability measure, Lemma \ref{ub-general} implies, for any convex $Q:$
$$\mu^+(\partial Q)\leq \frac{2n||f||_{\infty}}{\sup_{t\in (0,||f||_{\infty})}tR_t}.$$
\end{remark}

\medskip

Lastly, we need the following Lemma.

\begin{lemma}\label{existence}
For a bounded absolutely continuous unimodule probability density $f$, and for any constant $\alpha\in (0,1)$, there exists a $t>0$ such that
$$|K_t|\in\left[\frac{1-\alpha}{||f||_{\infty}},\frac{1+\alpha}{||f||_{\infty}}\right].$$
\end{lemma}
\begin{proof} Firstly, let $\tau$ such that
$$\mu(K_{\tau})\geq 1-\alpha;$$
note that such a choice always exists, since $\mu(K_{t})\rightarrow_{t\rightarrow 0} 1$. Observe that we always have $\mu(K_{\tau})\leq |K_{\tau}|\cdot ||f||_{\infty}$, and thus
$$|K_{\tau}|\geq \frac{1-\alpha}{||f||_{\infty}}.$$
In case $|K_{\tau}|\leq \frac{1+\alpha}{||f||_{\infty}}$, we let $t=\tau$ and the proof is finished. Alternatively, suppose $|K_{\tau}|\geq \frac{1+\alpha}{||f||_{\infty}}$.

Next, let $s=\frac{||f||_{\infty}}{1+\alpha}$. By continiuity, $K_s$ is a convex set with non-empty interior. Observe that
$$|K_{s}|\leq \frac{\mu(K_s)}{s} \leq\frac{1}{s}\leq\frac{1+\alpha}{||f||_{\infty}}.$$
In case $|K_{s}|\geq \frac{1-\alpha}{||f||_{\infty}}$, we let $t=s$ and the proof is finished. Alternatively, suppose $|K_{s}|\leq \frac{1-\alpha}{||f||_{\infty}}$.

Since $|K_{\tau}|\geq \frac{1+\alpha}{||f||_{\infty}}$, and $|K_{s}|\leq \frac{1-\alpha}{||f||_{\infty}}$, by continuity and monotonicity of the $|K_t|$ function, there exists a $t\in [\tau, s]$ such that $|K_t|=\frac{1}{||f||_{\infty}}$, and hence the Lemma follows in this case as well.
\end{proof}

\medskip
\medskip

\textbf{Proof of Theorem \ref{ub}.} Let $f$ be a unimodule density with level sets $K_t(f)$. Pick an $\alpha>0$ and let $t$ be such that
\begin{equation}\label{t}
|K_t(f)|\cdot ||f||_{\infty}\in [1-\alpha,1+\alpha].
\end{equation}

Let $T$ be such a linear volume preserving transformation that the ellipsoid of maximal volume inside the set $TK_t(f)$ is a ball of radius $R$. Consider $\mu_T$, the push forward of $\mu$ under $T$, with density by $f_T(x)=f(T^{-1}x)$. Note that 
$$K_t(f_T)=T K_t(f),$$
and also that $||f||_{\infty}=||f_T||_{\infty}$. Therefore, in view of (\ref{t}), and the fact that $T$ is volume preserving,
\begin{equation}\label{newt}
|K_t(f_T)|\cdot ||f_T||_{\infty}\in [1-\alpha,1+\alpha].
\end{equation}
By Lemma \ref{ub-general} and (\ref{newt}), we estimate, for every convex set $Q$:
\begin{equation}\label{estimate}
\mu_T^+(\partial Q)\leq \frac{(2+\alpha)n}{R_t(f_T)}.
\end{equation}
Note that $R_t(f_T)\geq R$, by our choice of the operator $T.$

Recall that a convex body is said to be in John's position if the ellipsoid of maximal volume contained in it is the unit ball. Ball \cite{ball_cube1} showed that if a convex body $L$ is in John's position, then $|L|\leq C_0^n$, for an absolute constant $C_0,$ with equality when $L$ is a simplex. Therefore, by our choice of $T$ and $t,$
\begin{equation}\label{finest}
R_t(f_T)\geq R\geq \frac{|K_t(f_T)|^{\frac{1}{n}}}{C_0}\geq \frac{(1-\alpha)^{\frac{1}{n}}||f_T||_{\infty}^{-\frac{1}{n}}}{C_0}.
\end{equation}
By letting $\alpha\rightarrow 0,$ we conclude, by (\ref{estimate}) and (\ref{finest}), that
$$\mu_T^+(\partial Q)\leq 2C_0n||f_T||_{\infty}^{\frac{1}{n}},$$
and hence the theorem is proved. $\square$

\begin{remark} Our argument shows that whenever $\mu$ is a measure with absolutely continuous unimodule density which is additionally even, we have
\begin{equation}\label{eq-cube}
\Gamma(\mu)\leq 4n||f||^{\frac{1}{n}}_{\infty}.
\end{equation}
This follows from the fact that for a symmetric convex set $Q$ in John's position, $|Q|\leq 2^n,$ with equality in the case $Q=B^n_{\infty}$, as was shown by Ball \cite{ball_cube1}. 

In the case when $\mu$ is uniform on $B^n_{\infty}$, one has $\Gamma(\mu)=2n||f||^{\frac{1}{n}}_{\infty}$. The precise optimal value of the constant in (\ref{eq-cube}) is not clear at the moment: for instance, does there exist an even unimodule measure $\mu$ with $\Gamma(\mu)>2n||f||^{\frac{1}{n}}_{\infty}$?
\end{remark}

\medskip
\medskip

\subsection{The case of isotropic log-concave measures.}

A measure $\mu$ on $\R^n$ is called log-concave if for every pair of Borel sets $A$ and $B,$ 
$$\mu(\lambda A+(1-\lambda)B)\geq \mu(A)^{\lambda}\mu(B)^{1-\lambda}.$$ 
In accordance with Borel's result, the density of $\mu$ has the form $f(x)=e^{-\varphi(x)},$ where $\varphi(x)$ is convex on $\R^n$. Let $X$ be a random vector on $\R^n$ distributed according to $\mu$. We say that $\mu$ is isotropic if $\E X=0$ and $\E\langle X,\theta\rangle^2=1$ for every $\theta\in\sfe.$ 

We refer the reader to Klartag \cite{Kl}, \cite{Kl2}, \cite{Kl3}, \cite{Kl4}, Klartag, Milman \cite{KM}, Paouris \cite{Pao1} or the books by Brazitikous, Giannopolous, Valettas, Vritsiou \cite{gian}, Koldobsky \cite{Kold}, and Artstein, Giannopolous, Milman \cite{AGMbook} for a comprehensive overview of the geometry of isotropic log-concave measures. 

To finish the proof of Proposition \ref{ub}, we shall need the following fact, known to the experts, which was proved e.g. by Klartag in \cite{Kl}, although it was not formulated explicitly there, and hence we outline the implication.

\begin{lemma}\label{klartag}
There exists an absolute constant $C_0>0$ such that for any isotropic log-concave measure $\mu$ with density $f$, letting $s=||f||_{\infty}e^{-C_0n},$ we have
$$R_{s}(f)\geq \frac{1}{10}.$$
\end{lemma}
\begin{proof} Firstly, by Corollary 5.3 from Klartag \cite{Kl}, there exists a constant $C_1>0$ such that for any $\alpha>C_1$,
$$\mu(K_{||f||_{\infty}e^{-\alpha n}}(f))\geq 1-e^{-\alpha n/8}.$$
Consequently, for $C_0=\max(C_1, 8\log 10),$ we have
$$\mu(K_{s}(f))\geq \frac{9}{10},$$
where $s=||f||_{\infty}e^{-C_0n}$. The argument involves integration in polar coordinates and asymptotic estimates for the arising integrals.

Secondly, by Lemma 5.4 from Klartag \cite{Kl}, for any convex $K$ with $\mu(K)\geq \frac{9}{10}$, one has $\frac{1}{10}B_2^n\subset K.$ The argument is based on comparing the inradius of a level set to its maximal section, and uses isotropicity. 

The combination of these facts yields the lemma. 
\end{proof}

\begin{remark} 
One may arrive to the conclusion of Lemma \ref{klartag} in a variety of ways. For example, one may use the analysis of Ball's bodies, and argue along the lines of Klartag, Milman \cite{KM}. Unfortunately, neither of the ways seem to allow to get $|K|\cdot||f||_{\infty}\leq C$ while reducing $R_s$ by only a constant factor. 
\end{remark}

Observe that Lemma \ref{klartag}, combined with Lemma \ref{ub-general} yields the upper bound of $e^{Cn}$ for $\Gamma(X)$, when $X$ is an isotropic log-concave vector. Such a bound is very rough; in order to get a polynomial estimate, we will need an additional application of log-concavity.

\begin{lemma}\label{scaling}
For a log-concave density $f$, for any $t>0$, and any $\lambda\in [0,1]$, there exists $y\in\R^n$ such that
$$K_{t^{\frac{1}{\lambda}}||f||_{\infty}}\subset \frac{1}{\lambda} K_{t||f||_{\infty}}+y.$$
\end{lemma}
\begin{proof} Let us denote by $x_{max}$ the (or any) point for which $f(x_{max})=||f||_{\infty}$. By log-conavity, for any $\lambda\in [0,1]$, 
$$f(\lambda x+(1-\lambda)x_{max})\geq f(x)^{\lambda}f(x_{\max})^{1-\lambda},$$
and therefore, for any $t>0$,
$$
K_{t^{\frac{1}{\lambda}}||f||_{\infty}}=\{x: f(x)^{\lambda}||f||_{\infty}^{1-\lambda}\geq t||f||_{\infty}\}\subset 
$$
\begin{equation}\label{logconc}
\{x: f(\lambda x+(1-\lambda)x_{max})\geq t||f||_{\infty}\}=\frac{1}{\lambda} K_{t||f||_{\infty}}+y, 
\end{equation}
for some vector $y.$ The Lemma follows.
\end{proof}

\begin{remark} The example of $f=C_n e^{-||x||_1}$ shows that Lemma \ref{scaling} is in fact sharp, and so is part of the estimate below in which it is used.
\end{remark}

\textbf{Proof of the Proposition \ref{ub-lc}.}  Letting $\lambda=\frac{1}{n}$ and $t=e^{-C_0}$ in Lemma \ref{scaling}, with $C_0$ from Lemma \ref{klartag}, we observe:
$$\frac{1}{10n}B_2^n+z\subset K_{e^{-C_0}||f||_{\infty}},$$
for some vector $z\in\R^n.$ Therefore, letting $s=e^{-C_0}||f||_{\infty},$ we have
\begin{equation}\label{R_t}
R_{s}(f)\geq \frac{1}{10n}.
\end{equation}
Observe also that
\begin{equation}\label{vol}
|K_s(f)|\cdot ||f||_{\infty}\leq \frac{\mu(K_s)}{e^{-C_0}}\leq e^{C_0}.
\end{equation}
Combining Lemma \ref{ub-general} with (\ref{R_t}) and (\ref{vol}), we get, for every convex body $Q$:
$$\mu^+(\partial Q)\leq n\cdot10n\cdot (e^{C_0}+1)=C'n^2. \square$$

\medskip

We conclude the subsection by formulating the following

\begin{question}\label{question}
Let $\mu$ be an isotropic log-concave measure with density $f$. Does there exist a level set $K_t$ of $\mu$ such that 
$$|K_t|\leq \frac{C_1}{||f||_{\infty}},$$
and $C_2 B_2^n+y\subset K_t$, for some absolute constants $C_1$ and $C_2$ and a vector $y$?
\end{question}

In case the answer to this question is affirmative, then Lemma \ref{ub-general} yields the bound $\Gamma(\mu)\leq Cn$ for all isotropic log-concave measures $\mu$. Such a bound is attained by uniform measures on polytopes with few facets, such as cube and simplex.

\begin{observation}\label{obs}
In view of the above discussion, the answer to the Question \ref{question} is affirmative in the case when $\mu$ is an isotropic log-concave measure which additionally
\begin{itemize}
\item is uniform on a convex set, or
\item has homogenuous level sets, in particular if it has density of the form $C_ne^{-||\cdot||}$, or
\item is 1-symmetric, that is, it is invariant with respect to the symmetries of the cube.
\end{itemize}
Hence, in all those cases we have the sharp estimate $\Gamma(\mu)\leq Cn.$
\end{observation}

\medskip

\subsection{Another estimate.} We formulate also a generalization of Lemma \ref{simple}, which could be useful for estimating the surface area of convex sets with large inradii.

\begin{proposition}\label{ballinside}
Let $Q$ be a convex body in $\R^n$, containing $R B_2^n$. Then for every absolutely continuous measure $\mu$ with ray-decreasing density, 
$$\mu^+(\partial Q)\leq \frac{n\mu(Q)}{R}.$$
\end{proposition}
\begin{proof}
Let $\mu$ have density $f(y)=e^{-\varphi(y)}$. For every $y\in\R^n$ consider a one-dimensional function $\varphi_y(t)=\varphi(\frac{y}{|y|}t)$. By our assumptions, this function is increasing.

Following an idea from Nazarov \cite{fedia}, consider a map $X:\partial Q\times \R^+\rightarrow \R^n$ defined as $X(y,t)=yt.$ The Jacobian of this map is $t^{n-1} \langle y,n_y\rangle$, where $n_y$ is the unit normal to $Q$ at $y$. Therefore, 
$$\mu(Q)=\int_{\R^n} e^{-\varphi(y)}dy=\int_{\partial Q}\int_0^{1} \langle y, n_y\rangle t^{n-1} e^{-\varphi(ty)} d\sigma(y) dt=$$
$$
\int_{\partial Q} e^{-\varphi(y)} \langle y, n_y\rangle \frac{\int_0^{|y|} t^{n-1} e^{\varphi_y(t)}dt}{|y|^n e^{-\varphi_y(|y|)}} d\sigma(y)\geq 
$$
\begin{equation}\label{system1}
\mu^+(\partial Q)\cdot \min_{y\in\partial Q} \langle y, n_y\rangle\cdot \min_{y\in\R^n} \frac{\int_0^{|y|} t^{n-1} e^{\varphi_y(t)}dt}{|y|^n e^{-\varphi_y(|y|)}}.
\end{equation}
Above we used the expression for the perimeter $\mu^+(\partial Q)=\int_{\partial Q} e^{-\varphi(y)}d\sigma(y).$ Note that $R B_2^n\subset Q$ implies 
\begin{equation}\label{ball-R}
\langle y, n_y\rangle\geq R.
\end{equation}
In addition, as $f$ is ray decreasing, we have $e^{-\varphi_y(t)}\geq e^{-\varphi_y(|y|)}$, for any $t\in [0, |y|]$. Therefore,
\begin{equation}\label{lbnd-y}
\frac{\int_0^{|y|} t^{n-1} e^{\varphi_y(t)}dt}{|y|^n e^{-\varphi_y(|y|)}}\geq \frac{\int_{0}^{|y|} t^{n-1}dt}{|y|^n}=\frac{1}{n}.
\end{equation}
Combining (\ref{system1}), (\ref{ball-R}) and (\ref{lbnd-y}), we arrive to the conclusion.
\end{proof}

\section{Appendix}

Firstly, we observe the following fact, which follows directly by compactness.

\begin{fact}\label{tech1}
Suppose $\mu$ is an absolutely continuous measure with bounded density, and $K$ is a convex set such that $\mu^+(\partial K)<\infty$. Then there exists a pair of positive numbers $c, C$ (that depends only on the density of $\mu$, the upper bound on $\mu^+(K),$ and the dimension) such that for every $\epsilon\in [0,c],$
$$\frac{1}{\epsilon}\mu\left((K+\epsilon B_2^n)\setminus K\right)<C.$$
\end{fact}

\begin{lemma}\label{justif}
In the proof of Theorem \ref{lb}, either the right hand side of (\ref{step1}) equals to the left hand side of (\ref{step2}), or the conclusion of Theorem \ref{lb} follows anyway.
\end{lemma}
\begin{proof}
Without loss of generality we assume that there exists a constant $C_n$, possibly depending on the dimension $n$ and the distribution $\mu,$ such that for every convex body $K\subset\R^n,$ we have $\mu^+(\partial K)\leq C_n:$ indeed, otherwise there is nothing to prove.

If the density of $\mu$ is unbounded, consider the probability measure $\tilde{\mu}$ with density 
$$\tilde{f}(x)=2f(x)\cdot 1_{\{f\leq m_f\}}(x),$$
where $m_f$ is such a number that 
$$\mu(\{x:\,f(x)\leq m_f\})=\mu(\{x:\,f(x)\geq m_f\})=\frac{1}{2}.$$
Then $\tilde{\mu}$ is an absolutely continuous probability measure with bounded density, and for every convex body $K,$ we have $\mu^+(\partial K)\geq\frac{1}{2}\tilde{\mu}^+(\partial K)$. Therefore, without loss of generality we may assume that the density of $\mu$ is bounded (or else pass to $\tilde{\mu}$).

In view of the above assumptions, we may use Fact \ref{tech1}, which enables us to apply the dominated convergence Theorem, and interchange the limit and the expectation $\E_Y$. By Fubbini's theorem, we may also interchange $\E_Y$ and $\E_X.$ Using dominated convergence one more time, we interchange $\E_X$ and the limit.
\end{proof}

\end{document}